\documentclass[11pt,reqno]{amsart}
\usepackage[T1]{fontenc}
\usepackage{lmodern}
\usepackage{amsmath,amssymb,amsthm,mathtools,booktabs}
\usepackage[expansion=false]{microtype}
\usepackage{tikz}
\usetikzlibrary{calc,fit,backgrounds}
\usepackage{placeins}
\tikzset{
  cf diagram/.style={x=12pt,y=-12pt,line width=0.45pt,
    line cap=butt,line join=miter,
    every node/.style={font=\small,inner sep=0pt,outer sep=0pt}},
  cf edge/.style={line width=0.5pt,shorten <=4pt,shorten >=4pt},
  cf component/.style={draw=black!55,line width=0.4pt,
    dash pattern=on 2pt off 1.5pt,rounded corners=3pt,
    inner sep=4pt,outer sep=0pt,fill=none}
}
\newlength{\cfpathlabelwidth}
\newlength{\cfpathstep}
\usepackage[hidelinks]{hyperref}
\newtheorem{theorem}{Theorem}[section]
\newtheorem{proposition}[theorem]{Proposition}
\newtheorem{lemma}[theorem]{Lemma}
\newtheorem{corollary}[theorem]{Corollary}
\theoremstyle{definition}
\newtheorem{example}[theorem]{Example}
\theoremstyle{remark}
\newtheorem{remark}[theorem]{Remark}
\numberwithin{equation}{section}
\newcommand{\Sch}{\mathfrak S}
\newcommand{\key}{\kappa}
\newcommand{\Sky}{\operatorname{Sky}}
\newcommand{\calB}{\mathcal B}
\newcommand{\CC}{\mathbb C}
\newcommand{\ZZ}{\mathbb Z}

\title[Column factorizations and Schubert key products]{Factorizations by column blocks and Schubert polynomials as products of key polynomials}
\author{Xuanying Han}
\address{Department of Mathematics, Tianjin University of Finance and Economics, Tianjin 300222, P.R. China}
\email{hanxuanying@stu.tjufe.edu.cn}
\author{Shunli Ma}
\address{Department of Mathematics, Tianjin University of Finance and Economics, Tianjin 300222, P.R. China}
\email{mashunli@stu.tjufe.edu.cn}
\author{Sophie C. C. Sun}
\address{Department of Mathematics, Tianjin University of Finance and Economics, Tianjin 300222, P.R. China}
\email{sophiesun@tjufe.edu.cn}
\date{}
\hypersetup{pdftitle={Factorizations by column blocks and Schubert polynomials as products of key polynomials},pdfauthor={Xuanying Han; Shunli Ma; Sophie C. C. Sun},pdfsubject={Algebraic combinatorics}}
\begin{document}
\raggedbottom
\begin{abstract}
We classify the Schubert polynomials that are products of key
polynomials and explicitly determine their finest key
factorization. Their indexing permutations are precisely those
avoiding ten classical patterns of length six, which form the
minimal pattern basis. More generally, the connected components
of the column interaction graph give the unique finest
factorization of a dual flagged Weyl character into nonconstant
diagram characters. Every other such factorization is obtained
by grouping these components. We also prove Makhija's conjecture
on products of elementary symmetric polynomials in initial
alphabets, allowing repeated alphabets. In this case, the
elementary factors are uniquely determined, up to order, by the
nonempty Rothe columns.
\end{abstract}
\maketitle

\section{Introduction}

Schubert polynomials represent Schubert classes in the cohomology
of flag varieties~\cite{FMS}, and each is a nonnegative integer
linear combination of key polynomials~\cite{RS}. A Schubert
polynomial is itself a key polynomial exactly when its indexing
permutation avoids $2143$; see Reiner and Shimozono~\cite{RS}
and~\cite[Theorem~3.3]{MS}. We ask when a Schubert polynomial
is a product of keys and how to determine its finest such
factorization.

To study these products, we use dual flagged Weyl characters,
which we call \emph{diagram characters}. For a diagram
$D=(D_1,\ldots,D_m)$, where $D_j$ is the set of occupied
rows in column $j$, write $\chi_D$ for its character.
Rothe and skyline diagrams give Schubert and key polynomials,
respectively~\cite{FMS}. The product of the individual column
characters is a coefficientwise upper bound for $\chi_D$.
Equality in this bound is precisely a factorization into
individual column characters. Fan and Guo~\cite{FG} characterized
equality for Schubert and key polynomials; in the Schubert case,
it holds exactly when the indexing permutation avoids $1423$
and $1432$. Peng, Lin, and Sun~\cite[Theorem~1.1]{LPS}
gave the criterion for arbitrary diagrams.

We extend this question from individual columns to blocks
of columns. The obstruction to separating two columns is
recorded by the \emph{column interaction graph} $\Gamma_D$:
its vertices are the nonempty indexed columns of $D$, and
two columns are joined when they share an empty row above
a row occupied in both. We prove that the characters of
the column blocks multiply to $\chi_D$ exactly when no
edge of $\Gamma_D$ joins different blocks.

It remains to check that column partitions account for all
factorizations into diagram characters. Here we use the fact
that $\chi_D$ determines the multiset of nonempty columns of
$D$, a consequence of~\cite[Theorem~1.4]{HMSS}. Every factorization
into nonconstant diagram characters must therefore come from
a column partition. Together with the block criterion, this
shows that the connected components give the unique finest
factorization, up to order; every other such factorization
is obtained by grouping the component factors.\footnote{The
blocks need not have disjoint row sets; compare the disjoint-row
tensor products in~\cite[Example~1.6]{Anderson}. For factorizations
of Ehrhart polynomials of Schubitopes, see Dou, Fan, and
Liu~\cite{DFL}.}

The key-product problem now reduces to recognizing which
component characters are keys. A component character is a
key polynomial precisely when its columns form a chain
under inclusion. For Rothe diagrams, this condition gives
the pattern classification below.

To state it, write $\Sch_w$ for a Schubert polynomial and
$\key_\alpha$ for the key indexed by a \emph{weak composition}
$\alpha$, a finite sequence of nonnegative integers.
All factors use the fixed variables $x_1,x_2,\ldots$,
without relabeling. A \emph{key product} is a finite product
of nonconstant keys; a single factor and the empty product
$1$ are allowed. For $w=w_1\cdots w_n\in S_n$, the Rothe
diagram $D(w)$ has columns
$D_j(w)=\{i:i<w^{-1}(j),\ w_i>j\}$.
Classical pattern containment means having a subsequence
with the same relative order as the pattern. Set
\begin{align}\label{eq:basisintro}
\calB=\{&135264,136254,153264,163254,315264,316254,\notag\\
        &142635,142653,143625,143652\}.
\end{align}

\begin{theorem}\label{thm:intro}
The Schubert polynomial $\Sch_w$ is a key product if and only
if $w$ avoids every pattern in $\calB$. The ten patterns
in~\eqref{eq:basisintro} form the minimal classical pattern basis.

For such a permutation, group the nonempty Rothe columns
by the connected components of $\Gamma_{D(w)}$. Let
$\alpha^{(1)},\ldots,\alpha^{(s)}$ record the numbers of boxes
in each row of the resulting subdiagrams. Then the unique
finest key factorization, up to order, is
\[
 \Sch_w=\prod_{t=1}^s\key_{\alpha^{(t)}}.
\]
\end{theorem}

The same column criterion also applies to products of elementary
symmetric polynomials. For $1\le b\le a$, the polynomial
$e_b(x_1,\ldots,x_a)$ is the character of the interval column
$\{a-b+1,\ldots,a\}$. An \emph{elementary product} is a finite
product of these polynomials, allowing repeated initial
alphabets, a single factor, and the empty product $1$.

If each initial alphabet is used at most once, the product
is a \emph{standard elementary monomial}. Hatam, Johnson,
Liu, and Macaulay~\cite{HJLM} gave determinantal formulas
for expansions of Schubert polynomials in this basis when
the indexing permutation avoids thirteen specified patterns.
Woodruff~\cite[Theorem~4]{Woodruff} proved that a Schubert
polynomial is a single standard elementary monomial exactly
when its permutation avoids $312$ and $1432$. For products
allowing repeated initial alphabets,
Makhija~\cite[Section~1.7]{Makhija} conjectured the four-pattern
criterion below and gave a pipe-dream argument for
sufficiency~\cite[Section~2]{Makhija}. We prove the converse
and determine the factors uniquely.

\begin{theorem}\label{thm:elementary}
The Schubert polynomial $\Sch_w$ is an elementary product
if and only if $w$ avoids $1423$, $1432$, $3142$, and $4132$,
as conjectured by Makhija~\cite[Section~1.7]{Makhija}.
In this case the factorization is unique, up to order, and is
\begin{equation}\label{eq:elementaryformula}
 \Sch_w=\prod_{D_j(w)\ne\varnothing}
 e_{|D_j(w)|}(x_1,\ldots,x_{\max D_j(w)}).
\end{equation}
\end{theorem}

In this case, Proposition~\ref{prop:recovery} forces the elementary
factors to come from the individual Rothe columns. The criterion for
separating these columns~\cite{FG,LPS} and the characterization
of interval Rothe columns~\cite[Proposition~2.3]{LMS} give the
four forbidden patterns. Each interval column determines its
elementary factor, which also proves uniqueness.

Section~\ref{sec:prelim} recalls the determinant model and
shows that a diagram character determines the multiset of
its nonempty columns. Section~\ref{sec:block} proves the criterion
for column blocks and the finest factorization.
Sections~\ref{sec:keys} and~\ref{sec:elementary} apply these
results to key and elementary products, respectively.

\section{Columns and diagram characters}\label{sec:prelim}
All vector spaces are over $\CC$. We use
$[N]=\{1,\ldots,N\}$ and $[a,b]=\{a,a+1,\ldots,b\}$,
with $N$ large enough to contain the rows of every diagram
under discussion. For row sets $A,B$, the difference
$A\setminus B$ consists of the rows belonging to $A$ but
not to $B$. Empty columns contribute $1$ and play
no part in column partitions. Copies of the same nonempty
column retain their separate indices.

\subsection{Allowed upward moves and column characters}
Allowed moves in a column shift its boxes upwards without
changing their order. For two row sets of size $b$,
$C=\{c_1<\cdots<c_b\}$ and $A=\{a_1<\cdots<a_b\}$,
we write $C\preceq A$ in \emph{Gale order} if $c_i\le a_i$
for every $i$, and call $C$ \emph{allowed} for $A$.
For diagrams $E=(E_1,\ldots,E_m)$ and
$D=(D_1,\ldots,D_m)$ with $|E_j|=|D_j|$ for every $j$,
write $E\preceq D$ if $E_j\preceq D_j$ for every $j$.
We record the row counts of $E$ by the monomial
\[
 x^E=\prod_j\prod_{i\in E_j}x_i.
\]
Thus the exponent of $x_i$ is the number of cells in row $i$ of $E$.

We encode these allowed row choices by minors of a generic
upper triangular $N\times N$ matrix $Y=(y_{ij})$.
We write $Y_A^C$ for its submatrix with row set $C$ and
column set $A$, both in increasing order. For a column $A$,
the determinants $\det(Y_A^C)$ with $C\preceq A$ are
called its \emph{allowed~minors}. For $C=A$, this is the
\emph{principal~minor} $\det(Y_A^A)$. For a diagram $D$,
taking one allowed minor from each column gives the products
that span the polynomial space
\[
 M_D=\operatorname{span}_{\CC}
 \left\{\prod_j\det(Y_{D_j}^{E_j}):E\preceq D\right\}.
\]
We grade by rows, assigning weight $x_i$ to $y_{ij}$.
The coefficient $[x^\beta]\chi_D$ is the dimension of the
part of $M_D$ having weight $x^\beta$.

For a single column,
\begin{equation}\label{eq:FA}
 F_A=\chi_{(A)}=\sum_{C\preceq A}x^C,
 \qquad x^C=\prod_{i\in C}x_i.
\end{equation}
For $C=\{c_1<\cdots<c_b\}\preceq
A=\{a_1<\cdots<a_b\}$, the minor $\det(Y_A^C)$ contains
$\prod_i y_{c_i,a_i}$ with coefficient $1$, hence is nonzero.
As different $C$ give different row weights, the allowed
minors are independent.

For several columns, the products of allowed minors need not
be independent. Since they span $M_D$, we obtain the
coefficientwise bound
\[
 \chi_D\le_{\mathrm{coeff}}\prod_jF_{D_j}
       =\sum_{E\preceq D}x^E.
\]
Each product of allowed minors is nonzero, so every monomial
on the right also occurs in $\chi_D$. Thus the two sides have
the same support, although their coefficients may differ.

For $D=(\{2\},\{2\})$, each column has the two allowed
minors $y_{12}$ and $y_{22}$. Consequently,
\[
 M_D=\operatorname{span}_{\CC}
 \{y_{12}^2,y_{12}y_{22},y_{22}^2\},\qquad
 \chi_D=x_1^2+x_1x_2+x_2^2.
\]
By contrast, $(F_{\{2\}})^2=(x_1+x_2)^2$ has coefficient
$2$ at $x_1x_2$. The two choices of column minors give the
same product $y_{12}y_{22}$, which contributes only once to
the character.

For characters with $0$--$1$ coefficients, see Guo, Lin,
and Peng~\cite{GLP}. For related bounds on Schubert
monomials, see Guo and Lin~\cite{GuoLin}.

Permuting the columns does not change $M_D$ or $\chi_D$.
When comparing diagrams, we place them in a common
$N\times N$ grid, enlarging $N$ and adding empty columns
if necessary. These operations leave their spaces and
characters unchanged.

\subsection{Skyline diagrams and key polynomials}
For a weak composition $\alpha\in\ZZ_{\ge0}^N$, the
\emph{skyline diagram} $\Sky(\alpha)$ has $\alpha_i$
left-justified cells in row $i$. Its character is
$\key_\alpha$, and $\key_0=1$. The nonempty columns of
$\Sky(\alpha)$ are the sets
\begin{equation*}
 \{i:\alpha_i\ge1\}\supseteq\{i:\alpha_i\ge2\}\supseteq\cdots.
\end{equation*}
These columns form an \emph{inclusion chain}, meaning that
they are pairwise comparable by inclusion. Conversely, any
inclusion chain of columns can be put in skyline order.
The number of boxes in each row gives the indexing composition.
Figure~\ref{fig:skyline} shows $\Sky(3,0,2,1)$, with column
sets $\{1,3,4\}\supseteq\{1,3\}\supseteq\{1\}$.

\begin{figure}[htbp]
\centering
\begin{tikzpicture}[cf diagram]
  \foreach \r in {1,...,4}
    \node[anchor=east] at (-0.95,\r-1) {$\r$};
  \foreach \c in {1,...,3}
    \node at (\c-1,-1.05) {$\c$};
  \foreach \c in {0,1,2}
    \fill[black!15] (\c-0.5,-0.5) rectangle (\c+0.5,0.5);
  \foreach \c in {0,1}
    \fill[black!15] (\c-0.5,1.5) rectangle (\c+0.5,2.5);
  \fill[black!15] (-0.5,2.5) rectangle (0.5,3.5);
  \draw (-0.5,-0.5) rectangle (2.5,0.5);
  \draw (0.5,-0.5)--(0.5,0.5) (1.5,-0.5)--(1.5,0.5);
  \draw (-0.5,1.5)--(1.5,1.5)--(1.5,2.5)--(0.5,2.5)
        --(0.5,3.5)--(-0.5,3.5)--cycle;
  \draw (0.5,1.5)--(0.5,2.5) (-0.5,2.5)--(0.5,2.5);
  \node at (1,4.4) {$\Sky(3,0,2,1)$};
\end{tikzpicture}
\caption{The skyline diagram $\Sky(3,0,2,1)$.}
\label{fig:skyline}
\end{figure}
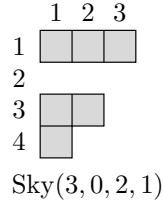

For example, $\key_{(1,0)}=x_1$ and
\mbox{$\key_{(0,1)}=x_1+x_2$}. Trailing zeros do not affect
a key.

For Rothe diagrams, $\chi_{D(w)}=\Sch_w$. These
identifications use the conventions of~\cite{FMS,LPS}.
When omitting empty Rothe columns, we keep the original
labels of the remaining columns.

\subsection{The interaction graph}
A pair of rows $i<j$ is a \emph{witness} for an interaction
between $A$ and $B$ if row $i$ is vacant in both columns
and row $j$ is occupied in both.
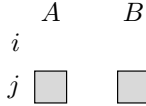
\[
\begin{tikzpicture}[cf diagram,baseline=(current bounding box.center)]
  \node at (0,-0.95) {$A$};
  \node at (2.6,-0.95) {$B$};
  \node[anchor=east] at (-0.95,0) {$i$};
  \node[anchor=east] at (-0.95,1.4) {$j$};
  \fill[black!15] (-0.5,0.9) rectangle (0.5,1.9);
  \fill[black!15] (2.1,0.9) rectangle (3.1,1.9);
  \draw (-0.5,0.9) rectangle (0.5,1.9);
  \draw (2.1,0.9) rectangle (3.1,1.9);
\end{tikzpicture}
\]
Only the two witness rows are shown; the columns need not be
consecutive in the diagram. A nonempty column $[a]$ is isolated
in $\Gamma_D$, since every row above an occupied row is also
occupied.

Peng, Lin, and Sun~\cite[Theorem~1.1]{LPS} characterized equality
in the coefficientwise bound for two columns:
\begin{equation}\label{eq:twocol}
 \chi_{(A,B)}=F_AF_B
 \quad\Longleftrightarrow\quad
 A\text{ and }B\text{ do not interact}.
\end{equation}

\subsection{Column multisets and character factorizations}
The following proposition is a direct consequence of
Hafner, M\'esz\'aros, Setiabrata, and St.~Dizier's
Theorem~1.4~\cite{HMSS}. We include a short proof and
its application to products of diagram characters.

\begin{proposition}\label{prop:recovery}
On a common $[N]$, the monomials occurring in $\chi_D$
determine the multiset of nonempty columns of $D$.
In particular, if
\[
 \chi_D=\chi_{E^{(1)}}\cdots\chi_{E^{(k)}},
\]
then the nonempty columns of the $E^{(t)}$ together are
exactly the nonempty columns of $D$, with multiplicity.
\end{proposition}
\begin{proof}
For each nonempty $A\subseteq[N]$, set
\[
 r_A(S)=\max_{C\preceq A}|C\cap S|
 \qquad(S\subseteq[N]).
\]
This is the maximum number of boxes of $A$ that can lie in
rows of $S$ after allowed moves. It is the rank function of the
Schubert matroid indexed by $A$. The functions $r_A$, over
nonempty $A\subseteq[N]$, are linearly independent by
\cite[Theorem~1.4]{HMSS}.

Let $m_D(A)$ be the multiplicity of a nonempty column $A$ in
$D$. Every monomial in $\chi_D$ is the weight of an allowed
choice of rows. For fixed $S$, choose a maximizing allowed
row set in each column. The resulting product of minors is
nonzero, so its weight occurs in $\chi_D$.
Grouping equal columns gives, for every $S\subseteq[N]$,
\[
 \begin{aligned}
 \max_{[x^\beta]\chi_D\ne0}\sum_{i\in S}\beta_i
   &=\sum_{D_j\ne\varnothing}r_{D_j}(S)\\
   &=\sum_{\varnothing\ne A\subseteq[N]}m_D(A)r_A(S).
 \end{aligned}
\]
Linear independence of the functions $r_A$ therefore determines
every multiplicity $m_D(A)$.

These maxima add under multiplication of polynomials with
nonnegative coefficients. Applying this to the asserted product
identity and using linear independence once more gives
$m_D(A)=\sum_t m_{E^{(t)}}(A)$ for every nonempty $A$,
as required.
\end{proof}

\section{Factorization by column blocks}\label{sec:block}
We prove that a partition of the columns gives a factorization
exactly when no interaction edge joins different blocks.
Together with Proposition~\ref{prop:recovery}, this yields the
unique finest factorization into nonconstant diagram
characters, up to order.

\subsection{Block factorization}
For a set $J$ of column indices, write $D[J]$ for the
subdiagram consisting of those columns.

\begin{theorem}\label{thm:block}
Let $\pi$ be a partition of the indices of the nonempty
columns of $D$. Then
\[
 \chi_D\le_{\mathrm{coeff}}\prod_{J\in\pi}\chi_{D[J]},
\]
with equality precisely when no edge of $\Gamma_D$ joins
distinct blocks, or equivalently when each block is a union
of connected components. If such an edge exists, at least one
coefficient is strictly smaller.
\end{theorem}

We begin with two blocks. Even if their minors share matrix
entries, the evaluations of the two blocks can be prescribed
independently when there is no edge between them. It suffices
to work with upper unitriangular matrices.

\begin{lemma}\label{lem:restriction}
For every fixed diagram $D$, setting $y_{jj}=1$ for all $j$
defines an injective linear map on $M_D$.
\end{lemma}
\begin{proof}
Put
\[
 d_j=\#\{u:j\in D_u\}.
\]
Every monomial $y^\nu$ occurring in $M_D$ has total degree
$d_j$ in the variables $y_{1j},\ldots,y_{jj}$. After setting
$y_{jj}=1$, the omitted exponent is still determined by
\[
 \nu_{jj}=d_j-\sum_{i<j}\nu_{ij}.
\]
Thus two distinct monomials cannot become the same monomial
after the substitution. Hence a nonzero element of $M_D$
remains nonzero after this substitution, proving injectivity.
\end{proof}

\begin{lemma}\label{lem:interpolation}
Let $P$ and $Q$ be two blocks of columns with no interaction
edge between them. For any upper unitriangular matrices $U,V$
there is an upper unitriangular matrix $W$ satisfying
\[
 f(W)=f(U)\quad(f\in M_P),\qquad
 g(W)=g(V)\quad(g\in M_Q).
\]
\end{lemma}
\begin{proof}
Set $W_{jj}=1$ and $W_{ij}=0$ for $i>j$.
For each $j$, determine the entries above the diagonal
in column $j$ of $W$ recursively, in the order
$i=j-1,j-2,\ldots,1$.
We choose each of these entries according to the following
rules. If a column of $P$
contains $j$ but not $i<j$, impose
\[
 (U^{-1}W)_{ij}=0.
\]
If the corresponding condition comes from $Q$, require
$(V^{-1}W)_{ij}=0$ instead. Both conditions could apply only if
columns of $P$ and $Q$ contained $j$ but omitted $i$, which
would give an interaction edge. If neither condition applies,
take $W_{ij}=0$. Matrix multiplication gives
\[
 (U^{-1}W)_{ij}
 =W_{ij}+(U^{-1})_{ij}
  +\sum_{i<\ell<j}(U^{-1})_{i\ell}W_{\ell j}.
\]
The coefficient of $W_{ij}$ is $1$, and the other entries of $W$
have already been chosen farther down column $j$. The same
holds with $V$ in place of $U$, so the construction is well
defined.

Now fix a column $A$ in $P$ and write $H=U^{-1}W$.
For $j\in A$ and $i\notin A$, we have $H_{ij}=0$:
by triangularity if $i>j$, and by construction if $i<j$.
The matrix $H_A^A$ is upper unitriangular. Hence, for every
$C\preceq A$,
\[
 \begin{aligned}
 W_A^C&=U_A^C H_A^A,\\
 \det(W_A^C)
   &=\det(U_A^C)\det(H_A^A)
    =\det(U_A^C).
 \end{aligned}
\]
Choose one allowed minor for each column of $P$.
Multiplying the corresponding equalities shows that the
resulting product has the same value at $W$ and $U$.
These products span $M_P$, so $f(W)=f(U)$ for every
$f\in M_P$. The same argument gives $g(W)=g(V)$ for
every $g\in M_Q$.
\end{proof}

\begin{example}\label{ex:interpolation}
Take $A=\{1,3\}$ and $B=\{2,3\}$. They both contain row $3$,
but there is no interaction edge between them. For
\[
 \begin{pmatrix}1&a&c\\0&1&b\\0&0&1\end{pmatrix},
\]
the allowed minors are $b,1$ for $A$ and $ab-c,a,1$
for $B$. The two sets of minors share the entry $b$. Nevertheless,
we can prescribe $b=\mu$ for $A$ and $a=\lambda$,
$ab-c=\rho$ for $B$ by taking $c=\lambda\mu-\rho$.
For example, $\mu=2$, $\lambda=3$, and $\rho=5$ give
$c=1$.
\end{example}

\begin{proof}[Proof of Theorem~\ref{thm:block}]
The claim is immediate for a single block. For two blocks $P,Q$,
choose homogeneous bases $f_1,\ldots,f_r$ of $M_P$ and
$g_1,\ldots,g_s$ of $M_Q$. Their pairwise products span
$M_D$. In each row degree, the coefficient of $\chi_P\chi_Q$
counts the corresponding pairs $(i,j)$, whereas the
coefficient of $\chi_D$ is the dimension of the span of
their products.
It follows that
\begin{equation}\label{eq:two-block-inequality}
 \chi_D\le_{\mathrm{coeff}}\chi_P\chi_Q.
\end{equation}
Equality holds if and only if the products $f_i g_j$
are linearly independent.

Suppose there is no edge from $P$ to $Q$, and consider
a linear relation
\[
 \sum_{i=1}^r\sum_{j=1}^s c_{ij}f_i(W)g_j(W)=0
\]
holding identically for upper triangular matrices $W$.
For arbitrary upper unitriangular matrices $U,V$, apply
this identity to a matrix $W$ supplied by
Lemma~\ref{lem:interpolation}, with
$f_i(W)=f_i(U)$ and $g_j(W)=g_j(V)$ for all $i,j$.
This gives
\[
 \sum_{i=1}^r\sum_{j=1}^s c_{ij}f_i(U)g_j(V)=0.
\]
By Lemma~\ref{lem:restriction}, both bases remain linearly
independent on upper unitriangular matrices. Fixing $V$ and
varying $U$ therefore gives $\sum_{j=1}^s c_{ij}g_j(V)=0$
for every $i$. The independence of the $g_j$, as $V$ varies,
then gives $c_{ij}=0$ for all $i,j$.
The products $f_i g_j$ are therefore linearly independent,
so equality holds in~\eqref{eq:two-block-inequality}.

Conversely, suppose that a column $A$ of $P$ interacts
with a column $B$ of $Q$. By~\eqref{eq:twocol}, there is
a nontrivial homogeneous linear relation among products
of an allowed minor of $A$ and an allowed minor of $B$.

For every indexed column other than the chosen columns
$A$ and $B$, fix its principal minor. Multiply each allowed
minor of $A$ by the fixed minors from $P$, and each allowed
minor of $B$ by those from $Q$. The resulting families lie
in $M_P$ and $M_Q$, respectively. They consist of homogeneous
polynomials and are linearly independent: the allowed minors
of each column are independent, and multiplication by a fixed
nonzero polynomial preserves this property. Extend these
families to homogeneous bases of $M_P$ and $M_Q$.

Multiplying the original relation by all the fixed minors
gives a nontrivial relation among products indexed by distinct
pairs of basis elements. Hence equality
in~\eqref{eq:two-block-inequality} is impossible, and the
inequality is strict in at least one row degree.

For more than two blocks, iterate this result.
Without crossing edges every split gives equality. Otherwise,
let $P$ consist of the columns in a block of $\pi$ containing
one endpoint of a crossing edge, and let $Q$ consist of the
columns in the other blocks.
The first split is strict. Applying the same inequality
within $Q$ gives
\[
 \chi_D
 <_{\mathrm{coeff}}\chi_P\chi_Q
 \le_{\mathrm{coeff}}
 \prod_{J\in\pi}\chi_{D[J]}.
\]
Here $<_{\mathrm{coeff}}$ means coefficientwise inequality
with strict inequality in at least one coefficient.
The description in terms of connected components follows.
\end{proof}

\subsection{The finest factorization}\label{sec:finest}
Recall that a factorization is finer than another when the
latter is obtained by grouping factors of the former.

\begin{corollary}\label{thm:canonical}
Let $K_1,\ldots,K_s$ be the connected components of $\Gamma_D$.
By Proposition~\ref{prop:recovery} and Theorem~\ref{thm:block},
the factorizations of $\chi_D$ into nonconstant diagram
characters are precisely those obtained by grouping these
components. Hence
\[
 \chi_D=\prod_{t=1}^s\chi_{D[K_t]}
\]
is the unique finest such factorization, up to order.
\end{corollary}

\begin{example}\label{ex:blocks}
For the columns $A,B$ in Example~\ref{ex:interpolation},
take $D=(A,A,B,B)$, as shown in Figure~\ref{fig:blocks}. The two copies
of $A$ interact, as do the two copies of $B$, but no $A$
interacts with a $B$. The dashed rounded boxes mark the two
components, with column indices $\{1,2\}$ and $\{3,4\}$.
The component factorization is
\[
 \chi_D=\chi_{(A,A)}\chi_{(B,B)},
\]
where
\[
 \chi_{(A,A)}=x_1^2(x_2^2+x_2x_3+x_3^2),
 \qquad
 \chi_{(B,B)}=s_{(2,2)}(x_1,x_2,x_3).
\]
Here $s_{(2,2)}$ is the Schur polynomial. Since each
component contains an edge, neither factor splits further
into the characters of its two columns.
\end{example}

\begin{figure}[htbp]
\centering
\begin{tikzpicture}[cf diagram]
  \foreach \r in {1,2,3}
    \node[anchor=east] at (-0.95,\r-1) {$\r$};
  \foreach \c in {1,2,3,4}
    \node at (\c-1,-1.05) {$\c$};
  \fill[black!15] (-0.5,-0.5) rectangle (1.5,0.5);
  \fill[black!15] (1.5,0.5) rectangle (3.5,1.5);
  \fill[black!15] (-0.5,1.5) rectangle (3.5,2.5);
  \draw (-0.5,-0.5) rectangle (1.5,0.5);
  \draw (0.5,-0.5)--(0.5,0.5);
  \draw (1.5,0.5) rectangle (3.5,1.5);
  \draw (2.5,0.5)--(2.5,1.5);
  \draw (-0.5,1.5) rectangle (3.5,2.5);
  \foreach \x in {0.5,1.5,2.5}
    \draw (\x,1.5)--(\x,2.5);

  \node (d1) at (7,0.2) {$D_1$};
  \node (d2) at (9.8,0.2) {$D_2$};
  \node (d3) at (12.6,0.2) {$D_3$};
  \node (d4) at (15.4,0.2) {$D_4$};
  \draw[cf edge] (d1.east)--(d2.west);
  \draw[cf edge] (d3.east)--(d4.west);
  \node (blockrows1) at (7,1.4) {$\{1,3\}$};
  \node (blockrows2) at (9.8,1.4) {$\{1,3\}$};
  \node (blockrows3) at (12.6,1.4) {$\{2,3\}$};
  \node (blockrows4) at (15.4,1.4) {$\{2,3\}$};
  \begin{scope}[on background layer]
    \node[cf component,inner xsep=3pt,fit=(d1)(d2)(blockrows1)(blockrows2)] {};
    \node[cf component,inner xsep=3pt,fit=(d3)(d4)(blockrows3)(blockrows4)] {};
  \end{scope}
  \node at (1.5,3.7) {$D$};
  \node at (11.2,3.7) {$\Gamma_D$};
\end{tikzpicture}
\caption{The diagram $D=(A,A,B,B)$ and its interaction graph.}
\label{fig:blocks}
\end{figure}
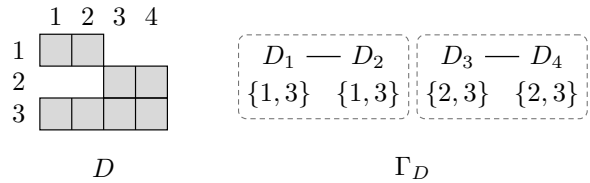
\FloatBarrier

\section{Products of key polynomials}\label{sec:keys}
We apply the criterion for factorization by column blocks to
key products. We first characterize the components whose
characters are single key polynomials, and then translate
the component condition for Rothe diagrams into pattern
avoidance.

\subsection{The column condition}
We begin with the criterion for a single key polynomial.
\begin{remark}\label{rem:singlekey}
A diagram character $\chi_E$ is a key polynomial if and only
if its nonempty columns form an inclusion chain. Indeed, the
columns of a chain can be arranged as a skyline, and the
converse follows from Proposition~\ref{prop:recovery}. The
indexing composition $\alpha$ is given by the row counts of
$E$, so $\chi_E=\key_\alpha$.
\end{remark}

\begin{example}\label{ex:splitkey}
A key polynomial can factor further. Let $D=\Sky(2,0,1)$.
Its columns are $D_1=\{1,3\}$ and $D_2=\{1\}$.
They satisfy $D_1\supsetneq D_2$ but do not interact.
In Figure~\ref{fig:splitkey}, the two dashed rounded boxes
mark the singleton components, and the inclusion sign
compares the two row sets. Hence
\[
 \key_{(2,0,1)}=x_1^2(x_2+x_3)
 =\key_{(1)}\key_{(1,0,1)}.
\]
\end{example}

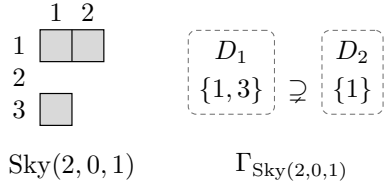
\begin{figure}[htbp]
\centering
\begin{tikzpicture}[cf diagram]
  \foreach \r in {1,...,3}
    \node[anchor=east] at (-0.95,\r-1) {$\r$};
  \node at (0,-1.05) {$1$};
  \node at (1,-1.05) {$2$};
  \fill[black!15] (-0.5,-0.5) rectangle (1.5,0.5);
  \fill[black!15] (-0.5,1.5) rectangle (0.5,2.5);
  \draw (-0.5,-0.5) rectangle (1.5,0.5);
  \draw (0.5,-0.5)--(0.5,0.5);
  \draw (-0.5,1.5) rectangle (0.5,2.5);
  \node (e1) at (5.5,0.2) {$D_1$};
  \node (e2) at (9.3,0.2) {$D_2$};
  \node (splitrows1) at (5.5,1.4) {$\{1,3\}$};
  \node (splitrows2) at (9.3,1.4) {$\{1\}$};
  \begin{scope}[on background layer]
    \node[cf component,fit=(e1)(splitrows1)] (splitcomponent1) {};
    \node[cf component,fit=(e2)(splitrows2)] (splitcomponent2) {};
  \end{scope}
  \node[anchor=base] at
    ($(splitrows1.base east)!0.5!(splitrows2.base west)$) {$\supsetneq$};
  \node at (0.5,3.8) {$\Sky(2,0,1)$};
  \node at (7.4,3.8) {$\Gamma_{\Sky(2,0,1)}$};
\end{tikzpicture}
\caption{A skyline diagram with two isolated columns.}
\label{fig:splitkey}
\end{figure}

\begin{theorem}\label{thm:keyproduct}
The character $\chi_D$ is a key product if and only if the columns
in every connected component of $\Gamma_D$ form an inclusion chain.
When this holds, let $K_1,\ldots,K_s$ be these components,
and let $\alpha_i^{(t)}$ count
the boxes in row $i$ of $D[K_t]$. Then
\begin{equation*}
 \chi_D=\prod_{t=1}^s\key_{\alpha^{(t)}}.
\end{equation*}
This is the unique finest key factorization, up to order;
indexing compositions are understood up to trailing zero parts.
\end{theorem}
\begin{proof}
A key factor is a skyline character, so its columns form
an inclusion chain by Remark~\ref{rem:singlekey}.
Proposition~\ref{prop:recovery} and
Corollary~\ref{thm:canonical} show that a connected component
cannot be split among different factors. Thus every component must be a
chain. Conversely, when the components are chains, arrange
the columns of each as a skyline. Their characters are the
claimed keys, and the component factorization is the unique
finest one.
\end{proof}

\subsection{The pattern condition}
Even if every edge joins comparable columns, the columns in a
connected component need not form an inclusion chain. Consider $w=142635$, whose nonempty Rothe columns are
\[
 D_2=\{2\},\qquad D_3=\{2,4\},\qquad D_5=\{4\}.
\]
Figure~\ref{fig:path} shows the path $D_2-D_3-D_5$.
The dashed rounded box encloses its single component.
Its neighboring columns satisfy
$D_2\subsetneq D_3\supsetneq D_5$, but the endpoints
are incomparable.

\begin{figure}[htbp]
\centering
\settowidth{\cfpathlabelwidth}{\small $D_2$}
\setlength{\cfpathstep}{\dimexpr(56.4pt+\cfpathlabelwidth+8pt)/2\relax}
\begin{tikzpicture}[cf diagram,x=16pt,y=-16pt,
  cf full grid/.style={line width=0.25pt,
    dash pattern=on 1.1pt off 1.1pt}]
  \fill[black!15] (0.5,0.5) rectangle (2.5,1.5);
  \fill[black!15] (1.5,2.5) rectangle (2.5,3.5);
  \fill[black!15] (3.5,2.5) rectangle (4.5,3.5);
  \foreach \k in {0,...,6} {
    \draw[cf full grid] (\k-0.5,-0.5)--(\k-0.5,5.5);
    \draw[cf full grid] (-0.5,\k-0.5)--(5.5,\k-0.5);
  }
  \foreach \r in {1,...,6}
    \node[anchor=east] at (-0.95,\r-1) {$\r$};
  \foreach \j in {1,...,6}
    \node at (\j-1,-1.05) {$D_{\j}$};
  \draw (0.5,0.5) rectangle (2.5,1.5);
  \draw (1.5,0.5)--(1.5,1.5);
  \draw (1.5,2.5) rectangle (2.5,3.5);
  \draw (3.5,2.5) rectangle (4.5,3.5);
  \node at (2.5,6.2) {$D(142635)$};
  \begin{scope}[xshift=144pt]
    \node (d2) at (0,1.6) {$D_2$};
    \node (d3) at ([xshift=\cfpathstep]d2.center) {$D_3$};
    \node (d5) at ([xshift=\cfpathstep]d3.center) {$D_5$};
    \draw[cf edge] (d2.east)--(d3.west);
    \draw[cf edge] (d3.east)--(d5.west);
    \node (pathrows2) at ([yshift=-20pt]d2.center) {$\{2\}$};
    \node (pathrows3) at ([yshift=-20pt]d3.center) {$\{2,4\}$};
    \node (pathrows5) at ([yshift=-20pt]d5.center) {$\{4\}$};
    \node[anchor=base] at
      ($(pathrows2.base east)!0.5!(pathrows3.base west)$) {$\subsetneq$};
    \node[anchor=base] at
      ($(pathrows3.base east)!0.5!(pathrows5.base west)$) {$\supsetneq$};
    \begin{scope}[on background layer]
      \node[cf component,fit=(d2)(d3)(d5)(pathrows2)(pathrows3)(pathrows5)] {};
    \end{scope}
    \node at (\cfpathstep,6.2) {$\Gamma_{D(142635)}$};
  \end{scope}
\end{tikzpicture}
\caption{The Rothe diagram $D(142635)$ and its interaction graph.}
\label{fig:path}
\end{figure}
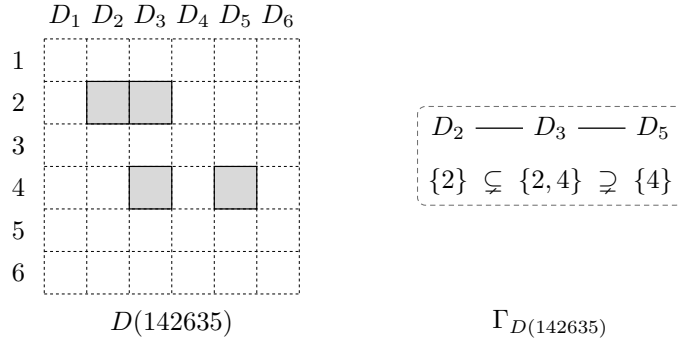
\FloatBarrier

\begin{lemma}\label{lem:short}
Every component of $\Gamma_D$ whose columns do not form an inclusion chain contains either an edge between incomparable columns, or a path $A-B-C$ with incomparable endpoints and
\[
 B\subseteq A\cap C\qquad\text{or}\qquad B\supseteq A\cup C.
\]
\end{lemma}
\begin{proof}
Choose incomparable columns $A,C$ having minimum graph
distance in their component and take a shortest path between
them. Distance one gives the first alternative. Otherwise,
every other pair of vertices on the path has smaller graph
distance than $A,C$ and is therefore comparable by minimality.
In particular, the internal columns form a chain, and each
is comparable with both $A$ and $C$.
If an internal column $B$ satisfied $A\subseteq B\subseteq C$
or $C\subseteq B\subseteq A$, then $A,C$ would be comparable,
a contradiction. Thus every internal column is either
contained in both endpoints or contains both.

These two types cannot both occur. Indeed, consecutive
internal columns of different types would give
$L\subseteq A\cap C$ and $H\supseteq A\cup C$ with an edge
$LH$. A row pair witnessing this edge is vacant in $H$
and occupied in $L$, hence vacant and occupied, respectively,
in both $A$ and $C$. This would make $A,C$ adjacent.

If the internal columns are all contained in $A\cap C$,
choose the largest one, $B$. Since $B\subseteq A\cap C$,
the row pair witnessing the first edge of the path also
witnesses $AB$; the last edge similarly witnesses $BC$.
If all internal columns contain $A\cup C$, choose the
smallest one instead. The same reasoning, with inclusions
reversed, gives $AB$ and $BC$.
\end{proof}

A diagram is \emph{northwest} if boxes at the northeast
and southwest corners of a rectangle force a box at the
northwest corner. Equivalently, when $A$ lies to the left
of $C$, every row of $C\setminus A$ is below every row of
$A$.

\begin{lemma}\label{lem:nwshort}
Let $D$ be a northwest diagram. A component of $\Gamma_D$ fails to be an inclusion chain if and only if it contains either an edge between incomparable columns, or a path of length two $A-B-C$ with nonadjacent incomparable endpoints and
\[
 B\supseteq A\cup C.
\]
\end{lemma}
\begin{proof}
By Lemma~\ref{lem:short}, it suffices to exclude
$B\subseteq A\cap C$ when the incomparable endpoints
are nonadjacent. Put $A$ to the left of $C$, and let
$i<j$ witness $AB$.
Then $j\in B\subseteq C$, whereas $i\notin A$. If $i$
were in $C$, it would belong to $C\setminus A$ above
the box at row $j$ of $A$, contrary to the northwest
property. Thus $i\notin C$ as well. The same pair
$i<j$ witnesses $AC$, so the endpoints cannot be
nonadjacent.
\end{proof}

\begin{remark}\label{rem:rothenw}
Rothe diagrams are northwest. For $j<k$, every row in
$D_k(w)\setminus D_j(w)$ lies below row $w^{-1}(j)$,
while every row of $D_j(w)$ lies above it. Consequently,
the components of $\Gamma_{D(w)}$ satisfy the
characterization in Lemma~\ref{lem:nwshort}.
\end{remark}

\begin{proof}[Proof of Theorem~\ref{thm:intro}]
The factorization formula and its uniqueness follow from
Theorem~\ref{thm:keyproduct}. It remains to prove the
avoidance criterion and minimality.

Suppose $\Sch_w$ is not a key product. By
Theorem~\ref{thm:keyproduct} and Remark~\ref{rem:rothenw},
there are incomparable Rothe columns $D_j,D_k$, $j<k$,
joined either by an edge or by a path of length two through
a column containing both. In each case, we find a subsequence
of $w$ whose pattern lies in $\calB$.

Choose entries $x,y$ with $w^{-1}(x)\in D_j\setminus D_k$
and $w^{-1}(y)\in D_k\setminus D_j$. Then $x>j$ and
$y>k$, with $x$ preceding $j$ and $y$ preceding $k$.
Since $y>k>j$ but $w^{-1}(y)\notin D_j$, $y$ follows $j$;
hence $x$ precedes $k$ and, since $w^{-1}(x)\notin D_k$,
satisfies $x<k$. Thus four selected entries occur as
\[
 x\,j\,y\,k,\qquad j<x<k<y.
\]

\smallskip
\noindent\textit{Case I: $D_j$ and $D_k$ are adjacent.}
Let the rows $w^{-1}(u)<w^{-1}(v)$ witness the edge.
Since $w^{-1}(v)\in D_j\cap D_k$, the entry $v$ is larger
than $k$ and precedes both $j$ and $k$. The entry $u$
precedes $v$, and hence $j$. Since $w^{-1}(u)\notin D_j$,
the definition of a Rothe column gives $u<j$.
Thus $u<j<k<v$. In particular,
\[
 u<j<x<k<\min(v,y).
\]
There are three positions for $x$ relative to $u,v$,
and two possible orders of the values $v,y$. These give
exactly six patterns:
\begin{center}
\begin{tabular}{@{}ccc@{}}
\toprule
Order in $w$ & $v<y$ & $y<v$\\
\midrule
$u\,x\,v\,j\,y\,k$ & $135264$ & $136254$\\
$u\,v\,x\,j\,y\,k$ & $153264$ & $163254$\\
$x\,u\,v\,j\,y\,k$ & $315264$ & $316254$\\
\bottomrule
\end{tabular}
\end{center}

\smallskip
\noindent\textit{Case II: $D_j-D_\ell-D_k$ is a path with
$D_j\cup D_k\subseteq D_\ell$ and nonadjacent endpoints.}
Because $w^{-1}(y)\in D_\ell$, the entry $y$ precedes
$\ell$ and is larger than it. Let the rows
$w^{-1}(u)<w^{-1}(v)$ witness the edge $D_j-D_\ell$.
The same argument as in Case~I shows that the entries $u,v$
occur in this order before both $j$ and $\ell$, with
$u<\min(j,\ell)<\max(j,\ell)<v$.
If $v>k$, the same witness rows would give an edge
$D_j-D_k$, contrary to the nonadjacency of the endpoints.
Also, $v\ne k$, since $v$ precedes $j$ whereas $k$
follows $j$. Hence $v<k$. We obtain
\[
 u<\min(j,\ell)<\max(j,\ell)<v<k<y.
\]
Since $j<v<k$ and $v$ precedes $j$,
$w^{-1}(v)\in D_j\setminus D_k$, so we may use $v$
in place of $x$.
The first four selected entries are $u,v,j,y$;
the positions of $\ell,k$ can be exchanged. Together
with the two possible orders of the values $j,\ell$,
this gives four patterns:
\begin{center}
\begin{tabular}{@{}ccc@{}}
\toprule
Order in $w$ & $j<\ell$ & $\ell<j$\\
\midrule
$u\,v\,j\,y\,\ell\,k$ & $142635$ & $143625$\\
$u\,v\,j\,y\,k\,\ell$ & $142653$ & $143652$\\
\bottomrule
\end{tabular}
\end{center}
Together, the two cases give the ten members of $\calB$.

Conversely, take any occurrence of a listed pattern and label
its selected entries according to the corresponding table.
Their order in $w$ and the relative order of their values give
the inequalities used above. In the first six cases,
$w^{-1}(x)\in D_j\setminus D_k$ and
$w^{-1}(y)\in D_k\setminus D_j$, while the rows
$w^{-1}(u)<w^{-1}(v)$ witness an edge between these columns.
In the last four, the rows $w^{-1}(v),w^{-1}(y)$ show that
the endpoint columns are incomparable, and the row pairs
$(w^{-1}(u),w^{-1}(v))$ and $(w^{-1}(u),w^{-1}(y))$
witness the two edges through $D_\ell$.
In both cases the indicated edges put two incomparable
columns in one component, so the chain condition fails.

Thus every failure of the chain condition contains a member
of $\calB$. Each member of $\calB$ itself fails the condition;
since they are distinct permutations of length six, none
contains another. This proves minimality.
\end{proof}

\section{Products of elementary symmetric polynomials}\label{sec:elementary}
The polynomial $e_b(x_1,\ldots,x_a)$ is the sum of the
squarefree monomials of degree $b$ in $x_1,\ldots,x_a$.
Since the factors have nonnegative coefficients, no variable
appearing in a factor can disappear from the product. We may
therefore work in a common finite alphabet.

\begin{lemma}\label{lem:interval}
For $1\le b\le a$,
\[
 F_{[a-b+1,a]}=e_b(x_1,\ldots,x_a).
\]
Conversely, if $F_A=e_b(x_1,\ldots,x_a)$ with $A\ne\varnothing$, then $A=[a-b+1,a]$.
\end{lemma}
\begin{proof}
If $C=\{c_1<\cdots<c_b\}\subseteq[a]$, then
$c_i\le a-b+i$. Hence the allowed row sets for
$[a-b+1,a]$ are precisely the subsets of $[a]$ of size $b$.
Formula~\eqref{eq:FA} gives the identity.

Conversely, equality forces the degrees and supports to
agree. Their greatest row sets in Gale order are $A$ and
$[a-b+1,a]$, respectively, so these sets are equal.
\end{proof}

\begin{proposition}\label{prop:elementarygeneral}
The character $\chi_D$ is an elementary product if and only if
every nonempty column of $D$ is an interval and $\Gamma_D$
is edgeless.
\end{proposition}
\begin{proof}
By Lemma~\ref{lem:interval}, every elementary factor
is the character of an interval column. If their product is
$\chi_D$, Proposition~\ref{prop:recovery} shows that the
interval columns, with multiplicities, are exactly the
nonempty columns of $D$. Thus
$\chi_D=\prod_{D_j\ne\varnothing}F_{D_j}$.
Applying Theorem~\ref{thm:block} to the partition into
singleton column blocks shows that $\Gamma_D$ is edgeless.

Conversely, if $\Gamma_D$ is edgeless,
Theorem~\ref{thm:block} gives
$\chi_D=\prod_{D_j\ne\varnothing}F_{D_j}$.
By Lemma~\ref{lem:interval}, these factors are elementary
when the columns are intervals.
\end{proof}

\begin{lemma}\label{lem:elementarypatterns}
For $w\in S_n$, the graph $\Gamma_{D(w)}$ is edgeless if and only if $w$ avoids $1423$ and $1432$, and every nonempty column of $D(w)$ is an interval if and only if $w$ avoids $3142$ and $4132$.
\end{lemma}
\begin{proof}
Suppose $D_j,D_k$ interact, with $j<k$.
By the definition of the Rothe columns, a witness consists
of rows $w^{-1}(u)<w^{-1}(v)$, where the entries $u,v$
occur before $j,k$ and satisfy $u<j<k<v$.
Depending on which of $j,k$ occurs first,
the selected entries form $1423$ or $1432$. Conversely,
either pattern provides just such a witness.

For the interval condition, suppose that a nonempty column
$D_j$ is not an interval. It then has a gap, a vacant
row between two occupied rows. Let $u,v,z$ be the entries
of $w$ in these rows, from top to bottom. By the definition of $D_j$, the occupied rows
lie above row $w^{-1}(j)$ and contain values larger than
$j$, while the intervening vacant row contains a value
smaller than $j$. Hence
$w^{-1}(u)<w^{-1}(v)<w^{-1}(z)<w^{-1}(j)$ and
$v<j<\min(u,z)$. The selected entries occur as
$u\,v\,z\,j$, giving pattern $3142$ or $4132$ according
to whether $u<z$ or $z<u$. Either pattern conversely
gives a gap in the column indexed by the last selected entry.
\end{proof}

The edgeless condition is equivalent to the equality
criterion of Fan and Guo~\cite{FG}. The characterization using
interval columns is due to Liu, M\'esz\'aros, and
St.~Dizier~\cite[Proposition~2.3]{LMS}.

\begin{proof}[Proof of Theorem~\ref{thm:elementary}]
Proposition~\ref{prop:elementarygeneral} and
Lemma~\ref{lem:elementarypatterns} give the avoidance
criterion. The factors are those in~\eqref{eq:elementaryformula}
by Lemma~\ref{lem:interval}. By Proposition~\ref{prop:recovery},
the interval columns are determined with multiplicities.
Each such column $A$ fixes $b=|A|$ and $a=\max A$,
so the factorization is unique up to order.
\end{proof}

Figure~\ref{fig:elementaryexamples} shows the Rothe diagrams
for $2431$, $1423$, and $3142$, from left to right. They
illustrate an elementary product and the two obstructions in
Proposition~\ref{prop:elementarygeneral}. The dashed boxes
group the nonempty column labels by connected component;
labels of empty columns remain unboxed. For $w=2431$, the
nonempty Rothe columns are
\[
 D_1=[1,3],\qquad D_3=\{2\}.
\]
These intervals do not interact, giving
\[
 \Sch_{2431}=e_3(x_1,x_2,x_3)e_1(x_1,x_2).
\]
For $1423$, the two nonempty columns are both
$\{2\}$, so they interact:
\[
 \Sch_{1423}=x_1^2+x_1x_2+x_2^2\ne(x_1+x_2)^2.
\]
For $3142$, the graph is edgeless, but
$D_2=\{1,3\}$ is not an interval.

\begin{figure}[htbp]
\centering
\begin{tikzpicture}[cf diagram,x=16pt,y=-16pt,
  cf full grid/.style={line width=0.25pt,
    dash pattern=on 1.1pt off 1.1pt}]
  \begin{scope}
    \fill[black!15] (-0.5,-0.5) rectangle (0.5,2.5);
    \fill[black!15] (1.5,0.5) rectangle (2.5,1.5);
    \foreach \k in {0,...,4} {
      \draw[cf full grid] (\k-0.5,-0.5)--(\k-0.5,3.5);
      \draw[cf full grid] (-0.5,\k-0.5)--(3.5,\k-0.5);
    }
    \foreach \r in {1,...,4}
      \node[anchor=east] at (-0.95,\r-1) {$\r$};
    \foreach \j in {1,...,4}
      \node (elementarya\j) at (\j-1,-1.05) {$D_{\j}$};
    \draw (-0.5,-0.5) rectangle (0.5,2.5);
    \draw (-0.5,0.5)--(0.5,0.5) (-0.5,1.5)--(0.5,1.5);
    \draw (1.5,0.5) rectangle (2.5,1.5);
    \begin{scope}[on background layer]
      \node[cf component,inner xsep=0.6pt,inner ysep=2pt,rounded corners=1.1pt,fit=(elementarya1)] {};
      \node[cf component,inner xsep=0.6pt,inner ysep=2pt,rounded corners=1.1pt,fit=(elementarya3)] {};
    \end{scope}
    \node at (1.5,4.2) {(a) $D(2431)$};
  \end{scope}
  \begin{scope}[xshift=96pt]
    \fill[black!15] (0.5,0.5) rectangle (2.5,1.5);
    \foreach \k in {0,...,4} {
      \draw[cf full grid] (\k-0.5,-0.5)--(\k-0.5,3.5);
      \draw[cf full grid] (-0.5,\k-0.5)--(3.5,\k-0.5);
    }
    \foreach \r in {1,...,4}
      \node[anchor=east] at (-0.95,\r-1) {$\r$};
    \foreach \j in {1,...,4}
      \node (elementaryb\j) at (\j-1,-1.05) {$D_{\j}$};
    \draw[line width=0.5pt]
      (elementaryb2.east)--(elementaryb3.west);
    \draw (0.5,0.5) rectangle (2.5,1.5);
    \draw (1.5,0.5)--(1.5,1.5);
    \begin{scope}[on background layer]
      \node[cf component,inner xsep=0.6pt,inner ysep=2pt,rounded corners=1.1pt,fit=(elementaryb2)(elementaryb3)] {};
    \end{scope}
    \node at (1.5,4.2) {(b) $D(1423)$};
  \end{scope}
  \begin{scope}[xshift=192pt]
    \fill[black!15] (-0.5,-0.5) rectangle (1.5,0.5);
    \fill[black!15] (0.5,1.5) rectangle (1.5,2.5);
    \foreach \k in {0,...,4} {
      \draw[cf full grid] (\k-0.5,-0.5)--(\k-0.5,3.5);
      \draw[cf full grid] (-0.5,\k-0.5)--(3.5,\k-0.5);
    }
    \foreach \r in {1,...,4}
      \node[anchor=east] at (-0.95,\r-1) {$\r$};
    \foreach \j in {1,...,4}
      \node (elementaryc\j) at (\j-1,-1.05) {$D_{\j}$};
    \draw (-0.5,-0.5) rectangle (1.5,0.5);
    \draw (0.5,-0.5)--(0.5,0.5);
    \draw (0.5,1.5) rectangle (1.5,2.5);
    \begin{scope}[on background layer]
      \node[cf component,inner xsep=0.6pt,inner ysep=2pt,rounded corners=1.1pt,fit=(elementaryc1)] {};
      \node[cf component,inner xsep=0.6pt,inner ysep=2pt,rounded corners=1.1pt,fit=(elementaryc2)] {};
    \end{scope}
    \node at (1.5,4.2) {(c) $D(3142)$};
  \end{scope}
\end{tikzpicture}
\caption{An elementary product and two obstructions.}
\label{fig:elementaryexamples}
\end{figure}
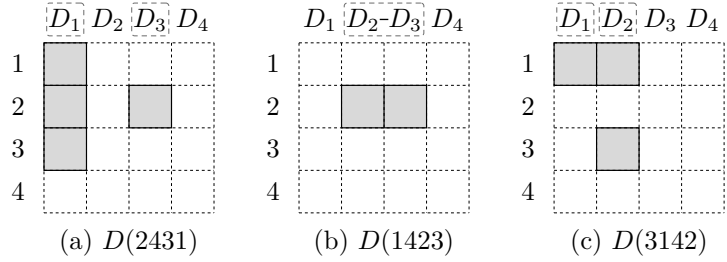

Allowing repeated initial alphabets matters. A
\emph{standard elementary monomial} (SEM) has the form
\[
 \prod_{i\ge1}e_{a_i}(x_1,\ldots,x_i),
 \qquad 0\le a_i\le i,
\]
where $e_0=1$ and all but finitely many $a_i$ vanish.
Woodruff~\cite[Definition~2 and Theorem~4]{Woodruff}
characterized single-SEM Schubert polynomials by avoidance
of $312$ and $1432$. In contrast,
$\Sch_{312}=x_1^2=e_1(x_1)^2$ is an elementary product
that uses the alphabet $\{x_1\}$ twice.

\section*{Acknowledgements}
This work was supported by the Natural Science Foundation of Tianjin City
(Grant No. 25JC-QNJC00250) and the National Natural Science Foundation of China
(Grant No. 12001398).

We used ChatGPT during the preparation of this work, primarily for writing and editing in English.

\end{document}